\documentclass[12pt]{article}

\usepackage{mathtools}
\usepackage{amsthm}
\usepackage{amsmath}
\usepackage{amsfonts}
\usepackage{listings}
\usepackage{longtable}
\usepackage{enumitem}
\usepackage{bm}

\newcommand{\Rm}[1]{\mathop{\mathrm{#1}}\nolimits}
\newcommand{\CC}{\mathcal{C}}
\newcommand{\FF}{\mathbb{F}}

\newcommand{\0}{\mathbf{0}}
\newcommand{\1}{\mathbf{1}}
\DeclareMathOperator{\wt}{wt}

\DeclarePairedDelimiter\floor{\lfloor}{\rfloor}

\newtheorem{thm}{Theorem}
\newtheorem{lem}{Lemma}

\newtheorem{prp}{Proposition}

\title{Classification of optimal quaternary Hermitian LCD codes of dimension $2$}
\author{Keita Ishizuka}
\date{}

\begin{document}
\maketitle

\begin{abstract}
Hermitian linear complementary dual codes are linear codes whose intersection with their Hermitian dual code is trivial.
The largest minimum weight among quaternary Hermitian linear complementary dual codes of dimension $2$ is known for each length. We give the complete classification of optimal quaternary Hermitian linear complementary dual codes of dimension $2$.
\end{abstract}

\section{Introduction}
\label{sec:introduction}

Let $\FF_4 \coloneqq \{ 0, 1, \omega, \omega^2 \}$ be the finite field of order four, where $\omega$ satisfies $\omega^2 + \omega + 1 = 0$. The conjugate of $x \in \FF_4$ is defined as $\overline{x} \coloneqq x^2$. A quaternary $[n, k, d]$ code is a linear subspace of $\FF_4^n$ with dimension $k$ and minimum weight $d$. Throughout this paper, we consider only linear quaternary codes and omit the term ``linear quaternary''.
Given a code $C$, a vector $c \in C$ is said to be a codeword of $C$. The weight of a codeword $c$ is denoted by $\wt(c)$. 

Let $u \coloneqq (u_1, u_2, \ldots, u_n),\ v \coloneqq (v_1, v_2, \ldots, v_n)$ be vectors of $\FF_4^n$. The Hermitian inner product is defined as $(u, v)_h \coloneqq u_1\overline{v_1} + u_2\overline{v_2} + \cdots + u_n\overline{v_n}$. Given a code $C$, the Hermitian dual code of $C$ is $C^{\perp h} \coloneqq \{ x \in \FF_4^n \mid (x, y)_h = 0 \text{ for all } y \in C \}$.
A generator matrix of the code $C$ is any matrix whose rows form a basis of $C$. Moreover, a generator matrix of the Hermitian dual code $C^{\perp h}$ is said to be a parity check matrix of $C$. Given a matrix $G$, we denote the transpose of $G$ by $G^T$ and the conjugate of $G$ by $\overline{G}$. 
Hermitian linear complementary dual codes, Hermitian LCD codes for short, are codes whose intersection with their Hermitian dual code is trivial. The concept of LCD codes was invented by Massey~\cite{massey1992linear} in 1992. LCD codes have been applied in data storage, communication systems and cryptography. For example, it is known that LCD codes can be used against side-channel attacks and fault injection attacks~\cite{carlet2016complementary}. We note that a code $C$ is a Hermitian LCD code if and only if its generator matrix $G$ satisfies $\det G\overline{G}^T \neq 0$~\cite{guneri2016quasi}.

Two codes $C, C'$ are equivalent if one can be obtained from the other by a permutation of the coordinates and a multiplication of any coordinate by a nonzero scalar. We denote the equivalence of two codes $C, C'$ by $C \simeq C'$.
Let $G, G'$ be generator matrices of two codes $C, C'$ respectively. It is known that $C \simeq C'$ if and only if $G$ can be obtained from $G'$ by an elementary row operation, a permutation of the columns and multiplication of any column by a nonzero scalar.

It was shown in~\cite{lu2015maximal} that the upper bound of the minimum weight of the Hermitian LCD $[n, 2, d]$ codes is given as follows:
\begin{equation}
    \label{eq:optDist}
    d \le \begin{cases}
        \floor*{\frac{4n}{5}}
        & \Rm{if\ } n \equiv 1, 2, 3 \pmod 5, \\
        \floor*{\frac{4n}{5}} - 1
        & \Rm{otherwise.}
    \end{cases}
\end{equation}
Also, it was proved that for all $n \neq 1$, there exists a Hermitian LCD $[n, 2, d]$ code which meets this upper bound.
We say that a Hermitian LCD $[n, 2, d]$ code is optimal if it meets this upper bound.
It was shown in~\cite{carlet2018linearcodes} that any code over $\FF_{q^2}$ is equivalent to some Hermitian LCD code for $q \ge 3$. Furthermore, it was proved in~\cite{lu2015maximal} that a Hermitian LCD code constructs a maximal-entanglement entanglement-assisted quantum error correcting code. Motivated by the results, we are concerned with the complete classification of optimal Hermitian LCD codes of dimension $2$.

This paper is organized as follows. In Section~\ref{sec:classification}, we present a method to construct optimal Hermitian LCD codes of dimension $2$, including all inequivalent codes. Also, a method to classify optimal Hermitian LCD codes of dimension $2$ is given. In Section~\ref{sec:optimal}, we classify optimal Hermitian LCD codes of dimension $2$. Up to equivalence, the complete classification of optimal Hermitian LCD codes of dimension $2$ is given. It is shown that all inequivalent codes have distinct weight enumerators, which is used for the classification.

\section{Classification Method}
\label{sec:classification}

Let $\0_n$ be the zero vector of length $n$ and $\1_n$ be the all-ones vector of length $n$. Let $(a_0, a_1, a_2, a_3, a_4, a_5) \neq (0, 0, 0, 0, 0, 0)$ be a tuple of nonnegative integers. We introduce the following notation:
\begin{align*}
G(a_0, a_1, a_2, a_3, a_4, a_5) \coloneqq
    \begin{pmatrix}
        1 & 0 & \0_{a_0} & \0_{a_1} & \1_{a_2} & \1_{a_3} & \1_{a_4} & \1_{a_5} \\
        0 & 1 & \0_{a_0} & \1_{a_1} & \0_{a_2} & \1_{a_3} & \omega \1_{a_4} & \omega^2 \1_{a_5}
    \end{pmatrix}.
\end{align*}
We denote by $C(a_0, a_1, a_2, a_3, a_4, a_5)$ the code whose generator matrix is $G(a_0, a_1, a_2, a_3, a_4, a_5)$. By the same argument as in~\cite{araya2019onthe}, we obtain the following lemma.

\begin{lem}
    \label{lem:distanceOne}
    Given a code $C$, define $C^{*} \coloneqq \{ (x, 0) \mid x \in C \}$. Let $\CC^{*}_{n, k}$ denote the set of all inequivalent Hermitian LCD $[n, k]$ codes C such that the minimum weight of $C^{\perp h}$ is $1$. Then there exists a set $\CC_{n-1, k}$ of all inequivalent Hermitian LCD $[n-1, k]$ codes such that $\CC^{*}_{n, k} = \{ C^{*} \mid C \in \CC_{n-1, k} \}$.
\end{lem}

We assume $a_0 = 0$ by Lemma~\ref{lem:distanceOne} and omit $a_0$. Furthermore, throughout this paper, we use the following notations:
\begin{equation}
    \label{eq:bolda}
    G(\bm{a}) \coloneqq G(a_1, a_2, a_3, a_4, a_5),\ C(\bm{a}) \coloneqq C(a_1, a_2, a_3, a_4, a_5),
\end{equation}
respectively, to save space.

\begin{prp}
    \label{prp:Cexists}
    Let $C$ be an $[n, 2, d]$ code. Then there exist nonnegative integers $a_1, a_2, a_3, a_4, a_5$ such that $C \simeq C(\bm{a})$ and $1 + a_2 + a_3 + a_4 + a_5 = d$.
\end{prp}

\begin{proof}
    Let $G$ be a generator matrix of the code $C$.
    By multiplying rows by some non-zero scalars, $G$ is changed to a generator matrix which consists only of the columns of $G(\bm{a})$.
    Permuting the columns, $G(\bm{a})$ is obtained from $G$. Hence it holds that $C \simeq C(\bm{a})$.
    Since the minimum weight of $C$ is $d$, we may assume that the first row of $G$ is a codeword with weight $d$, which yields $1 + a_2 + a_3 + a_4 + a_5 = d$. 
 \end{proof}

Given a code $C(\bm{a})$, we may assume without loss of generality that $G(\bm{a})$ satisfies
\begin{equation}
\label{eq:assumption}
1 + a_2 + a_3 + a_4 + a_5 = d,
\end{equation}
by Proposition~\ref{prp:Cexists}. This assumption on a generator matrix reduces computations later.

\begin{lem}
    \label{lem:sufnec}
    Let $C$ be an $[n, 2, d]$ code $C(\bm{a})$. Then $C$ is a Hermitian LCD code if and only if $C$ satisfies the following conditions:
    \begin{align}
        &a_1 = n - d - 1, \label{sufnec:a_1} \\
        &a_2 \leq n - d - 1, \label{sufnec:a_2} \\
        &a_3, a_4, a_5 \leq n - d, \label{sufnec:a_3a_4a_5} \\
        &\begin{cases}
            a_3 + a_4 + a_5 + a_3a_4 + a_4a_5 + a_5a_3 \not \equiv 0 \pmod 2 & \Rm{if} d \Rm{is} \Rm{even,} \\
            a_3a_4 + a_4a_5 + a_5a_3 \not \equiv n - d \pmod 2 & \Rm{otherwise.}
        \end{cases} \label{sufnec:LCD}
    \end{align}
\end{lem}

\begin{proof}
    Suppose $C$ is a Hermitian LCD $[n, 2, d]$ code.
    Let $G$ be a generator matrix of the code $C$. Let $r_1, r_2$ be the first and second rows of $G$ respectively.
    The number of columns of $G$ equals to the length $n$. Thus, it holds that
    \begin{equation}
        \label{eq:colEqLength}
        2 + a_1 + a_2 + a_3 + a_4 + a_5 = n.
    \end{equation}
    Since the minimum weight of $C$ is $d$, the following holds: $\wt(r_2) \ge d,\ \wt(r_1 + r_2) \ge d,\ \wt(r_1 + \omega r_2) \ge d,\ \wt(r_1 + \omega^2 r_2) \ge d$. By~\eqref{eq:assumption}, we have $\wt(r_1) = d$. Substituting~\eqref{eq:colEqLength} in each equation, we obtain~\eqref{sufnec:a_1} through~\eqref{sufnec:a_3a_4a_5}. 

    The code $C$ is a Hermitian LCD code if and only if
    \begin{equation*}
        \det{G\overline{G}^T} = (r_1, r_1)_h(r_2, r_2)_h - (r_2, r_1)_h(r_1, r_2)_h \neq 0,
    \end{equation*}
    where 
    \begin{align*}
        (r_1, r_1)_h &= 1 + a_2 + a_3 + a_4 + a_5 = d, \\
        (r_1, r_2)_h &= a_3 + \omega a_5 + \omega^2 a_4 = a_3 + a_4 + \omega (a_4 + a_5), \\
        (r_2, r_1)_h &= a_3 + \omega a_4 + \omega^2 a_5 = a_3 + a_5 + \omega (a_4 + a_5), \\
        (r_2, r_2)_h &= 1 + a_1 + a_3 + a_4 + a_5.
    \end{align*}
    Here we regard $n, d, a_3, a_4, a_5$ are elements of $\FF_4$.
    Therefore,~\eqref{sufnec:a_1} through~\eqref{sufnec:LCD} hold if $C$ is a Hermitian LCD $[n, 2, d]$ code and vice versa.
\end{proof}

Given an $[n, 2, d]$ code $C(\bm{a})$, we define the following:
\begin{align}
b_i &\coloneqq (n - d) - a_i \text{ for } 1 \le i \le 5. \label{eq:b_i}
\end{align}

\begin{lem}
    \label{lem:sufnecDash}
    Let $C$ be an $[n, 2, d]$ code $C(\bm{a})$. Then $C$ is a Hermitian LCD code if and only if $C$ satisfies the following conditions:
    \begin{align}
        &b_1 = 1, \label{sufnecDash:b_1} \\
        &b_2 \ge 1, \label{sufnecDash:upperbound} \\
        &b_3, b_4, b_5 \ge 0, \nonumber \\
        &\begin{cases}
            b_3 + b_4 + b_5 + b_3b_4 + b_4b_5 + b_5b_3 \not \equiv 0 \pmod 2 & \Rm{if} d \Rm{is} \Rm{even}, \\
            b_3b_4 + b_4b_5 + b_5b_3 \not \equiv 0 \pmod 2 & \Rm{otherwise.}
        \end{cases} \nonumber
    \end{align}
\end{lem}

\begin{proof}
    The result follows from Lemma~\ref{lem:sufnec}.
\end{proof}

Given an $[n, 2, d]$ code $C(\bm{a})$, we define the following:
\begin{align}
\Delta &\coloneqq 4n - 5d. \label{eq:DeltaDef}
\end{align}

\begin{lem}
    \label{lem:delta}
    Let $C$ be a Hermitian LCD $[n, 2, d]$ code. Then $C$ is optimal if and only if the value of $\Delta$ with respect to $n$ is given as follows:
    \begin{equation}
        \Delta = \begin{cases}
            5 &\Rm{if\ } n \equiv 0 \pmod 5, \\
            4 &\Rm{if\ } n \equiv 1 \pmod 5, \\
            3 &\Rm{if\ } n \equiv 2 \pmod 5, \\
            2 &\Rm{if\ } n \equiv 3 \pmod 5, \\
            6 &\Rm{if\ } n \equiv 4 \pmod 5. \\
        \end{cases} \nonumber
    \end{equation}
\end{lem}
\begin{proof}
    The result follows from~\eqref{eq:DeltaDef}.
\end{proof}

\begin{lem}
    \label{lem:deltaBound}
    Let $C$ be a code $C(\bm{a})$. If $C$ is a Hermitian LCD code, then it holds that
    \begin{equation}
        0 \le b_3, b_4, b_5 \le \Delta. \nonumber
    \end{equation}
\end{lem}
\begin{proof}
    Substituting $b_2, b_3, b_4, b_5, \Delta$ in~\eqref{eq:assumption}, we obtain
    \begin{equation}
        \label{eq:b_2}
        b_2 = \Delta + 1 - (b_3 + b_4 + b_5).
    \end{equation}
    Combining with~\eqref{sufnecDash:upperbound}, we obtain $b_3 + b_4 + b_5 \le \Delta$. Since $0 \le b_3, b_4, b_5$, it follows that $0 \le b_3, b_4, b_5 \le \Delta$.
\end{proof}

\begin{lem}
    \label{lem:equiva_3a_4a_5}
    \begin{align}
        &C(\bm{a}) \simeq C(a_1, a_2, a_5, a_3, a_4). \label{eq:equiva_3a_4a_5}
    \end{align}
\end{lem}

\begin{proof}
    Multiply the second row of $G(\bm{a})$ by $\omega$. Permuting the columns, the result follows. Note that we may assume that the nonzero entry of a column is $1$, provided that the entry of the other column is $0$.
\end{proof}

By Lemma~\ref{lem:equiva_3a_4a_5}, we may assume $a_3 \ge a_4, a_5$. Notice that $a_3 \ge a_4, a_5$ if and only if $b_3 \le b_4, b_5$ by~\eqref{eq:b_i}.

\section{Optimal Hermitian LCD Codes of Dimension $2$}
\label{sec:optimal}

By Lemmas \ref{lem:sufnecDash} through \ref{lem:equiva_3a_4a_5}, it
suffices to calculate all $b_3, b_4, b_5$ satisfying
\begin{align}
    &0 \le b_3 \le b_4, b_5 \le \Delta, \label{eq:reducedb_3b_4b_5} \\
    &\begin{cases}
        b_3 + b_4 + b_5 + b_3b_4 + b_4b_5 + b_5b_3 \neq 0 & \Rm{if} d \Rm{is} \Rm{even}, \\
        b_3b_4 + b_4b_5 + b_5b_3 \neq 0 & \Rm{otherwise,}
    \end{cases} \label{eq:reducedLCD}
\end{align}
in order to obtain optimal Hermitian LCD codes of dimension $2$, including all inequivalent codes.
Notice that $b_1, b_2$ are obtained by~\eqref{sufnecDash:b_1},~\eqref{eq:b_2} respectively.
Our computer search found all integers $b_3, b_4, b_5$ satisfying~\eqref{eq:reducedb_3b_4b_5} and~\eqref{eq:reducedLCD}.
This calculation was done by Magma~\cite{bosma1997magma}. Recall that $a_1, a_2, a_3, a_4, a_5$ are obtained from $b_1, b_2, b_3, b_4, b_5$ by~\eqref{eq:b_i}.
For optimal Hermitian LCD codes of dimension $2$, the integers $a_1, a_2, a_3, a_4, a_5$ are listed in Table~\ref{tab:allcodes}, where the rows are in lexicographical order with respect to $a_1, a_2, a_3, a_4, a_5$, and $m$ is a nonnegative integer.

\small
\begin{longtable}{llll}
% First head
\caption{Optimal Hermitian LCD codes of dimension $2$} \label{tab:allcodes} \\
$n$ & Code &  $(a_1, a_2, a_3, a_4, a_5)$ & $m$ \\
\hline \endfirsthead

% Headers
$n$ & Code &  $(a_1, a_2, a_3, a_4, a_5)$ & $m$\\
\hline \endhead

% Footers
\endfoot

% The last footer
\hline
\endlastfoot

    $n = 5m $
    & $C_{5m, 1}$ & $(m, m, m, m, m-2)$ & $m \ge 2$ \\
    & $C_{5m, 2}$ & $(m, m, m, m-2, m)$ & $m \ge 2$ \\
    & $C_{5m, 3}$ & $(m, m-1, m+1, m, m-2)$ & $m \ge 2$ \\
    & $C_{5m, 4}$ & $(m, m-1, m+1, m-2, m)$ & $m \ge 2$ \\
    & $C_{5m, 5}$ & $(m, m-1, m, m, m-1)$ & $m \ge 1$ \\
    & $C_{5m, 6}$ & $(m, m-1, m, m-1, m)$ & $m \ge 1$ \\
    & $C_{5m, 7}$ & $(m, m-2, m, m, m)$ & $m \ge 2$ \\
    & $C_{5m, 8}$ & $(m, m-3, m+1, m, m)$ & $m \ge 3$ \\
    \hline

    $n = 5m+1$
    & $C_{5m+1, 1}$ & $(m, m, m+1, m, m-2)$ & $m \ge 2$ \\
    & $C_{5m+1, 2}$ & $(m, m, m+1, m-2, m)$ & $m \ge 2$ \\
    & $C_{5m+1, 3}$ & $(m, m, m, m, m-1)$ & $m \ge 1$ \\
    & $C_{5m+1, 4}$ & $(m, m, m, m-1, m)$ & $m \ge 1$ \\
    & $C_{5m+1, 5}$ & $(m, m-1, m+1, m+1, m-2)$ & $m \ge 2$ \\
    & $C_{5m+1, 6}$ & $(m, m-1, m+1, m, m-1)$ & $m \ge 1$ \\
    & $C_{5m+1, 7}$ & $(m, m-1, m+1, m-1, m)$ & $m \ge 1$ \\
    & $C_{5m+1, 8}$ & $(m, m-1, m+1, m-2, m+1)$ & $m \ge 2$ \\
    & $C_{5m+1, 9}$ & $(m, m-2, m+1, m, m)$ & $m \ge 2$ \\
    & $C_{5m+1, 10}$ & $(m, m-3, m+1, m+1, m)$ & $m \ge 3$ \\
    & $C_{5m+1, 11}$ & $(m, m-3, m+1, m, m+1)$ & $m \ge 3$ \\
    \hline
    
    $n = 5m + 2$
    & $C_{5m+2, 1}$ & $(m, m, m, m, m)$ & $m \ge 0$ \\
    & $C_{5m+2, 2}$ & $(m, m-1, m+1, m, m)$ & $m \ge 1$ \\
    \hline

    $n = 5m + 3$
    & $C_{5m+3, 1}$ & $(m, m-1, m+1, m+1, m)$ & $m \ge 1$ \\
    & $C_{5m+3, 2}$ & $(m, m, m+1, m, m)$ & $m \ge 0$ \\
    & $C_{5m+3, 3}$ & $(m, m-1, m+1, m, m+1)$ & $m \ge 1$ \\
    \hline
    
    $n = 5m + 4$
    & $C_{5m+4, 1}$ & $(m+1, m+1, m+1, m+1, m-2)$ & $m \ge 2$ \\
    & $C_{5m+4, 2}$ & $(m+1, m+1, m+1, m, m-1)$ & $m \ge 1$ \\
    & $C_{5m+4, 3}$ & $(m+1, m+1, m+1, m-1, m)$ & $m \ge 1$ \\
    & $C_{5m+4, 4}$ & $(m+1, m+1, m+1, m-2, m+1)$ & $m \ge 2$ \\
    & $C_{5m+4, 5}$ & $(m+1, m+1, m+2, m+1, m-3)$ & $m \ge 3$ \\
    & $C_{5m+4, 6}$ & $(m+1, m+1, m+2, m-1, m-1)$ & $m \ge 1$ \\
    & $C_{5m+4, 7}$ & $(m+1, m+1, m+2, m-3, m+1)$ & $m \ge 3$ \\
    & $C_{5m+4, 8}$ & $(m+1, m, m+1, m, m)$ & $m \ge 0$ \\
    & $C_{5m+4, 9}$ & $(m+1, m, m+2, m+1, m-2)$ & $m \ge 2$ \\
    & $C_{5m+4, 10}$ & $(m+1, m, m+2, m+2, m-3)$ & $m \ge 3$ \\
    & $C_{5m+4, 11}$ & $(m+1, m, m+2, m, m-1)$ & $m \ge 1$ \\
    & $C_{5m+4, 12}$ & $(m+1, m, m+2, m-1, m)$ & $m \ge 1$ \\
    & $C_{5m+4, 13}$ & $(m+1, m, m+2, m-2, m+1)$ & $m \ge 2$ \\
    & $C_{5m+4, 14}$ & $(m+1, m, m+2, m-3, m+2)$ & $m \ge 3$ \\
    & $C_{5m+4, 15}$ & $(m+1, m-1, m+1, m+1, m)$ & $m \ge 1$ \\
    & $C_{5m+4, 16}$ & $(m+1, m-1, m+1, m, m+1)$ & $m \ge 1$ \\
    & $C_{5m+4, 17}$ & $(m+1, m-1, m+2, m+1, m-1)$ & $m \ge 1$ \\
    & $C_{5m+4, 18}$ & $(m+1, m-1, m+2, m-1, m+1)$ & $m \ge 1$ \\
    & $C_{5m+4, 19}$ & $(m+1, m-2, m+2, m+1, m)$ & $m \ge 2$ \\
    & $C_{5m+4, 20}$ & $(m+1, m-2, m+2, m+2, m-1)$ & $m \ge 2$ \\
    & $C_{5m+4, 21}$ & $(m+1, m-2, m+2, m, m+1)$ & $m \ge 2$ \\
    & $C_{5m+4, 22}$ & $(m+1, m-2, m+2, m-1, m+2)$ & $m \ge 2$ \\
    & $C_{5m+4, 23}$ & $(m+1, m-3, m+2, m+1, m+1)$ & $m \ge 3$ \\
    & $C_{5m+4, 24}$ & $(m+1, m-4, m+2, m+1, m+2)$ & $m \ge 4$ \\
    & $C_{5m+4, 25}$ & $(m+1, m-4, m+2, m+2, m+1)$ & $m \ge 4$ \\
\end{longtable}

\begin{lem}
    \label{lem:equiva_2anda_3}
    Suppose $a_3$ is a positive integer. Then
    \begin{align}
        &C(\bm{a}) \simeq C(a_1, a_3-1, a_2+1, a_5, a_4). \label{eq:equiva_2anda_3}
    \end{align}
\end{lem}

\begin{proof}
    Add the second row of $G(\bm{a})$ to the first row. Permuting the columns, the result follows. Recall that $C(\bm{a})$ is defined in~\eqref{eq:bolda}.
\end{proof}

\begin{lem}
    \label{lem:equiva_1a_2}
    Suppose $1 + a_1 + a_3 + a_4 + a_5 = d$. Then
    \begin{align}
        &C(\bm{a}) \simeq C(a_2, a_1, a_3, a_5, a_4). \label{eq:equiva_1a_2}
    \end{align}
\end{lem}

\begin{proof}
    Interchange the first row and the second row of $G(\bm{a})$. Permuting the columns, the result follows.
\end{proof}

\begin{table}
    \begin{center}
        \small
        \caption{Equivalent optimal Hermitian LCD codes of dimension $2$}
        \begin{tabular}{lll} \\
            $n$ & Code \\ \hline
            $n = 5m$
            & $C_{5m, 7} \simeq_{\ref{eq:equiva_2anda_3}, \ref{eq:equiva_3a_4a_5}} C_{5m, 6} \simeq_{\ref{eq:equiva_3a_4a_5}} C_{5m, 5}$ \\
            & $C_{5m, 8} \simeq_{\ref{eq:equiva_2anda_3}, \ref{eq:equiva_3a_4a_5}, \ref{eq:equiva_2anda_3}} C_{5m, 3} \simeq_{\ref{eq:equiva_2anda_3}} C_{5m, 2} \simeq_{\ref{eq:equiva_3a_4a_5}} C_{5m, 1} \simeq_{\ref{eq:equiva_2anda_3}} C_{5m, 4}$ \\
            \hline

            $n = 5m + 1$
            & $C_{5m+1, 9} \simeq_{\ref{eq:equiva_2anda_3}, \ref{eq:equiva_3a_4a_5}, \ref{eq:equiva_2anda_3}} C_{5m+1, 6} \simeq_{\ref{eq:equiva_2anda_3}} C_{5m+1, 4} \simeq_{\ref{eq:equiva_3a_4a_5}} C_{5m+1, 3} \simeq_{\ref{eq:equiva_2anda_3}} C_{5m+1, 7}$ \\
            & $C_{5m+1, 11} \simeq_{\ref{eq:equiva_2anda_3}, \ref{eq:equiva_3a_4a_5}, \ref{eq:equiva_2anda_3}} C_{5m+1, 5} \simeq_{\ref{eq:equiva_2anda_3}, \ref{eq:equiva_3a_4a_5}} C_{5m+1, 1} \simeq_{\ref{eq:equiva_1a_2}} C_{5m+1, 2}$ \\
            & $\simeq_{\ref{eq:equiva_3a_4a_5}, \ref{eq:equiva_2anda_3}} C_{5m+1, 8} \simeq_{\ref{eq:equiva_2anda_3}, \ref{eq:equiva_3a_4a_5}, \ref{eq:equiva_2anda_3}} C_{5m+1, 10}$ \\
            \hline

            $n = 5m + 2$
            & $C_{5m+2, 1} \simeq_{\ref{eq:equiva_2anda_3}} C_{5m+2, 2}$ \\
            \hline

            $n = 5m + 3$
            & $C_{5m+3, 1} \simeq_{\ref{eq:equiva_2anda_3}, \ref{eq:equiva_3a_4a_5}} C_{5m+3, 2} \simeq_{\ref{eq:equiva_3a_4a_5}, \ref{eq:equiva_2anda_3}} C_{5m+3, 3}$ \\
            \hline

            $n = 5m + 4$
            & $C_{5m+4, 8} \simeq_{\ref{eq:equiva_3a_4a_5}, \ref{eq:equiva_2anda_3}} C_{5m+4, 16} \simeq_{\ref{eq:equiva_3a_4a_5}} C_{5m+4, 15}$ \\
            & $C_{5m+4, 6} \simeq_{\ref{eq:equiva_3a_4a_5}, \ref{eq:equiva_2anda_3}} C_{5m+4, 22} \simeq_{\ref{eq:equiva_3a_4a_5}} C_{5m+4, 20}$ \\
            & $C_{5m+4, 21} \simeq_{\ref{eq:equiva_2anda_3}, \ref{eq:equiva_3a_4a_5}, \ref{eq:equiva_2anda_3}} C_{5m+4, 17} \simeq_{\ref{eq:equiva_2anda_3}, \ref{eq:equiva_3a_4a_5}, \ref{eq:equiva_2anda_3}} C_{5m+4, 12}$ \\
            & $\simeq_{\ref{eq:equiva_2anda_3}} C_{5m+4, 2} \simeq_{\ref{eq:equiva_1a_2}} C_{5m+4, 3} \simeq_{\ref{eq:equiva_2anda_3}} C_{5m+4, 11}$ \\
            & $\simeq_{\ref{eq:equiva_2anda_3}, \ref{eq:equiva_3a_4a_5}, \ref{eq:equiva_3a_4a_5}, \ref{eq:equiva_2anda_3}} C_{5m+4, 19} \simeq_{\ref{eq:equiva_2anda_3}, \ref{eq:equiva_3a_4a_5}, \ref{eq:equiva_3a_4a_5}, \ref{eq:equiva_2anda_3}} C_{5m+4, 18}$ \\
            & $C_{5m+4, 23} \simeq_{\ref{eq:equiva_2anda_3}, \ref{eq:equiva_3a_4a_5}, \ref{eq:equiva_2anda_3}} C_{5m+4, 9}$ \\
            & $\simeq_{\ref{eq:equiva_2anda_3}} C_{5m+4, 4} \simeq_{\ref{eq:equiva_1a_2}} C_{5m+4, 1} \simeq_{\ref{eq:equiva_2anda_3}} C_{5m+4, 13}$ \\
            & $C_{5m+4, 24} \simeq_{\ref{eq:equiva_2anda_3}, \ref{eq:equiva_3a_4a_5}, \ref{eq:equiva_2anda_3}} C_{5m+4, 10} \simeq_{\ref{eq:equiva_2anda_3}, \ref{eq:equiva_3a_4a_5}} C_{5m+4, 5} \simeq_{\ref{eq:equiva_1a_2}} C_{5m+4, 7}$ \\
            & $\simeq_{\ref{eq:equiva_3a_4a_5}, \ref{eq:equiva_2anda_3}} C_{5m+4, 14} \simeq_{\ref{eq:equiva_2anda_3}, \ref{eq:equiva_3a_4a_5}, \ref{eq:equiva_2anda_3}} C_{5m+4, 25}$\\
            \hline
        \end{tabular}
        \label{tab:classificationMLarge}
    \end{center}
\end{table}

By Lemmas~\ref{lem:equiva_3a_4a_5} through~\ref{lem:equiva_1a_2}, we have the equivalences among some codes listed in Table~\ref{tab:allcodes}, which are displayed in Table~\ref{tab:classificationMLarge}. Note that $C \simeq_i C'$ denotes the two codes $C, C'$ are equivalent by $(i)$. Also, $C \simeq_{i, j} C'$ denotes that, given two codes $C, C'$, there exists a code $C''$ such that $C \simeq_i C'' \simeq_j C'$.

Table~\ref{tab:weightEnum} gives the weight enumerators of representatives in Table~\ref{tab:classificationMLarge}. The weight enumerator is given by $1 + 3y^{\wt(r_1)} + 3y^{\wt(r_2)} + 3y^{\wt(r_1 + r_2)} + 3y^{\wt(r_1 + \omega r_2)} + 3y^{\wt(r_1 + \omega^2 r_2)}$, where $r_1, r_2$ be the first and second rows of $G(\bm{a})$ respectively. Since the weight enumerators are distinct, the codes in Table~\ref{tab:weightEnum} are inequivalent.
Table~\ref{tab:classificationLength} gives the classification of optimal Hermitian LCD codes of dimension $2$, with the case where $n$ is so small that some codes in Table~\ref{tab:allcodes} do not exist.

Recall that we have assumed $a_0 = 0$ by Lemma~\ref{lem:distanceOne}. It follows from~\eqref{eq:optDist} that there exists an optimal Hermitian LCD $[n, 2]$ code $C$ such that the minimum weight of $C^{\perp h}$ equals to $1$ if and only if $n \equiv 4 \pmod 5$.
Consequently, we obtain the following theorem.

\begin{thm}
    \label{thm:allcodes}
    \begin{enumerate}[label=$(\mathrm{\roman*})$]
        \setlength \itemsep{0em}
        \item Up to equivalence, there exist two optimal Hermitian LCD $[5m, 2, 4m-1]$ codes for every integer $m$ with $m \ge 2$.
        \item Up to equivalence, there exist two optimal Hermitian LCD $[5m+1, 2, 4m]$ codes for every integer $m$ with $m \ge 2$.
        \item Up to equivalence, there exists unique optimal Hermitian LCD $[5m+2, 2, 4m+1]$ code for every integer $m$ with $m \ge 0$.
        \item Up to equivalence, there exists unique optimal Hermitian LCD $[5m+3, 2, 4m+2]$ code for every integer $m$ with $m \ge 0$.
        \item Up to equivalence, there exist six optimal Hermitian LCD $[5m+4, 2, 4m+2]$ codes for every integer $m$ with $m \ge 3$. One of them is the code such that the minimum weight of the Hermitian dual code is $1$.
    \end{enumerate}
\end{thm}

\begin{table}
    \begin{center}
        \caption{Weight enumerators of representatives}
        \small
        \begin{tabular}{lll} \\
            $n$ & Code & Weight Enumerator \\ \hline

            $n = 5m$
            & $C_{5m, 7}$ & $1 + 3y^{4m-1} + 9y^{4m} + 3y^{4m+1}$ \\
            & $C_{5m, 8}$ & $1 + 6y^{4m-1} + 6y^{4m} + 3y^{4m+2}$ \\
            \hline

            $n = 5m + 1$
            & $C_{5m+1, 9}$ & $1 + 6y^{4m} + 6y^{4m+1} + 3y^{4m+2}$ \\
            & $C_{5m+1, 1}$ & $1 + 9y^{4m} + 3y^{4m+1} + 3y^{4m+3}$ \\
            \hline

            $n = 5m + 2$
            & $C_{5m+2, 1}$ & $1 + 6y^{4m+1} + 6y^{4m+2} + 3y^{4m+3}$ \\
            \hline

            $n = 5m + 3$
            & $C_{5m+3, 1}$ & $1 + 9y^{4m+2} + 6y^{4m+3}$ \\ \hline

            $n = 5m + 4$
            & $C_{5m+4, 8}$ & $1 + 3y^{4m+2} + 6y^{4m+3} + 6y^{4m+4}$ \\
            & $C_{5m+4, 6}$ & $1 + 9y^{4m+2} + 6y^{4m+5}$ \\
            & $C_{5m+4, 21}$ & $1 + 6y^{4m+2} + 3y^{4m+3} + 3y^{4m+4} + 3y^{4m+5}$ \\
            & $C_{5m+4, 23}$ & $1 + 6y^{4m+2} + 6y^{4m+3} + 3y^{4m+6}$ \\
            & $C_{5m+4, 24}$ & $1 + 9y^{4m+2} + 3y^{4m+3} + 3y^{4m+7}$ \\
            \hline
        \end{tabular}
        \label{tab:weightEnum}
    \end{center}
\end{table}

\begin{table}
    \begin{center}
        \caption{Classification of optimal Hermitian LCD codes of dimension $2$}
        \small
        \begin{tabular}{lll} \\
            $n$ & $m$ & Code \\ \hline

            $n = 5m$
            & $m = 1$ & $C_{5m, 7}$ \\
            & $m \ge 2$ & $C_{5m, 7}, C_{5m, 8}$ \\
            \hline

            $n = 5m + 1$
            & $m = 1$ & $C_{5m+1, 9}$ \\
            & $m \ge 2$ & $C_{5m+1, 9}, C_{5m, 1}$ \\
            \hline

            $n = 5m + 2$
            & $m \ge 0$ & $C_{5m+2, 1}$ \\
            \hline

            $n = 5m + 3$
            & $m \ge 0$ & $C_{5m+3, 1}$ \\
            \hline

            $n = 5m + 4$
            & $m = 0$ & $C_{5m+4, 8}$ \\
            & $m = 1$ & $C_{5m+4, 8}, C_{5m+4, 6}, C_{5m+4, 21}$ \\
            & $m = 2$ & $C_{5m+4, 8}, C_{5m+4, 6}, C_{5m+4, 21}, C_{5m+4, 23}$ \\
            & $m \ge 3$ & $C_{5m+4, 8}, C_{5m+4, 6}, C_{5m+4, 21}, C_{5m+4, 23}, C_{5m+4, 24}$ \\
            \hline
        \end{tabular}
        \label{tab:classificationLength}
    \end{center}
\end{table}

\section*{Acknowledgements}
The author would like to thank supervisor Professor Masaaki Harada for introducing the problem, useful discussions and his encouragement.

\end{document}